\documentclass[a4paper,11pt,leqno]{amsart}

\usepackage{amsmath}
\usepackage{amsthm}
\usepackage{amsfonts}
\usepackage{amssymb}

\usepackage{tikz}
\usetikzlibrary{arrows}
\usepackage{enumerate}
\usetikzlibrary{calc}
\usepackage{colonequals}
\usepackage[all]{xy}
\usepackage{todonotes}
\usepackage{verbatim}

\newcommand{\inv}{^{-1}}
\newcommand{\id}{\operatorname{id}}

\newcommand{\loc}{\operatorname{loc}}
\newcommand{\eps}{\varepsilon}

\newcommand{\R}{\mathbb R}
\newcommand{\bS}{\mathbb S}

\newcommand{\abs}[1]{\lvert#1\rvert}
\newcommand{\norm}[1]{\lVert#1\rVert}

\newtheorem{theorem}{Theorem}[section]
\newtheorem*{theorem*}{Theorem}{\bf}{\it}
\newtheorem{proposition}[theorem]{Proposition}
\newtheorem*{proposition*}{Proposition}{\bf}{\it}
\newtheorem{lemma}[theorem]{Lemma}
\newtheorem*{lemma*}{Lemma}{\bf}{\it}
\newtheorem{corollary}[theorem]{Corollary}
\theoremstyle{definition}
\newtheorem{definition}[theorem]{Definition}
\newtheorem*{definition*}{Definition}
\theoremstyle{remark}
\newtheorem{remark}[theorem]{Remark}
\numberwithin{equation}{section}

\numberwithin{equation}{section}

\begin{document}

\title[Compactness of branch]{Mappings of finite distortion: compactness of the branch set}

\author{Aapo Kauranen}
\address{Department of Mathematical Analysis, Sokolovska 83, Praha 8, 186 75, Charles University in Prague \and 
Department de Matem\`atiques, Universitat Aut\`onoma de Barcelona, 08193, 
Bellaterra (Barcelona), Spain}
\email{aapo.p.kauranen@gmail.com}
\thanks{A.K.
acknowledges financial support from the Spanish Ministry of Economy and 
Competitiveness, through the ``Marí\'{\i}a de Maeztu'' Programme for Units of 
Excellence in R\&D (MDM-2014-0445) and MTM-2016-77635-P (MICINN, Spain).}

\author{Rami Luisto}
\address{Department of Mathematics and Statistics, P.O. Box 35, FI-40014 University of Jyv\"askyl\"a, Finland \and
Department of Mathematical Analysis, Sokolovska 83, Praha 8, 186 75, Charles University in Prague}
\email{rami.luisto@gmail.com}
\thanks{R.L.\
  was supported by the Finnish Academy of Science and Letters.
}

\author{Ville Tengvall}
\address{Department of Mathematics and Statistics, P.O. Box 35, FI-40014 University of Jyv\"askyl\"a, Finland}
\email{ville.tengvall@jyu.fi}
\thanks{V.T. was supported by the Academy of Finland Project 277923.}

\subjclass[2010]{57M12, 30C65}
\keywords{Quasiregular mappings, branch set, branched cover, mappings of finite distortion.}
\date{\today}

\begin{abstract}
  We show that an entire branched cover of finite distortion cannot have a compact branch set
  if its distortion satisfies a certain asymptotic growth condition.
  We furthermore show that this bound is 
  strict by constructing an entire, continuous, open and discrete mapping of finite distortion which is piecewise smooth,
  has a branch set homeomorphic to $(n-2)$-dimensional torus and distortion arbitrarily close to the asymptotic bound.
\end{abstract}

\maketitle

\section{Introduction}
\label{sec:Intro}

A mapping $f \in W^{1,1}(\Omega, \mathbb{R}^n)$, defined on an open set $\Omega \subset \mathbb{R}^n$ with $n \ge 2$, is called a \emph{mapping of finite distortion} if
\begin{itemize}
\item[(FO-1)] $J_f \in L_{\loc}^1(\Omega)$, and
\item[(FO-2)] $J_f > 0$ a.e. on the set where $\abs{Df(x)} > 0$,
\end{itemize}
where
$$\abs{Df(x)} \colonequals \sup_{\abs{v}=1} \abs{Df(x)v} \quad \text{and} \quad J_{f}(x) \colonequals \det Df(x)$$ 
are the operator norm and the Jacobian determinant of the differential matrix $Df(x)$, respectively. With such a mapping we associate a Lebesgue measurable \emph{outer distortion function} $K_O(\cdot, f) \colon \Omega \to [1, \infty]$ defined as follows
\begin{align*}
&K_{O}(x,f) = \left\{ \begin{array}{ll}
\frac{\abs{Df(x)}^{n}}{J_{f}(x)}, & \textrm{if $J_{f}(x)>0$}\\
1, & \textrm{otherwise.}
\end{array} \right.
\end{align*}
If the outer distortion function is essentially bounded we call the mapping \emph{quasiregular}. In addition, homeomorphic quasiregular mappings are called \emph{quasiconformal mappings}. For the basic background of these three different mapping classes we recommend the monographs \cite{Astala-Iwaniec-Martin, HenclKoskela, IwaniecMartin, Rickman-book, VaisalaBook, Vuorinen} for the reader. In this paper we are especially interested in the local and global injectivity of entire quasiregular mappings and mapping of finite distortion.

A basic example of an entire quasiregular mapping which is not quasiconformal is the \emph{winding mapping} $\omega \colon \mathbb{R}^n \to \mathbb{R}^n$ defined in cylindrical coordinates by the formula
\begin{align*}
\omega(r\cos \theta, r\sin \theta, x_3, \ldots, x_n) = (r\cos 2\theta, r\sin 2\theta, x_3, \ldots, x_n),
\end{align*}
where $r \ge 0$ and $\theta \in \mathbb{R}$. For this mapping we see that its \emph{branch set}
\begin{align*}
B_{\omega} = \{ x \in \mathbb{R}^n : x_1 = 0, x_2 = 0 \} \, ,
\end{align*}
id est, the set where the mapping fails to be a local homeomorphism, is nonempty and therefore the winding mapping cannot be neither a global homeomorphism nor a quasiconformal mapping. On the other hand, by Zorich's global homeomorphism theorem \cite{Zorich} an entire quasiregular mapping with an empty branch set is always a quasiconformal mapping. Thus, the study of quasiconformality of entire quasiregular mappings can be reduced to the study of their branch sets.

More generally, the structure of the branch set of a mapping is connected to the topology and geometry of the mapping itself. 
The branch set is of particular interest
in the study of \emph{branched covers}, i.e.\ continuous, open and discrete mappings; see e.g.\ 
\cite{ChurchHemmingsen1,ChurchHemmingsen2,ChurchHemmingsen3}, \cite{Edmonds78,Edmonds79,BernsteinEdmonds78,BernsteinEdmonds78} 
and \cite{AaltonenPankka}.
Note that all nonconstant quasiregular mappings are branched covers by the fundamental theorem of Reshetnyak, see e.g. \cite[Theorem I.4.1]{Rickman-book}. Also nonconstant mappings of finite distortion are branched covers under certain exponential integrability conditions on the distortion function, see \cite{HenclKoskela}. Alternatively, one may replace the assumption on the exponentially integrable distortion by a suitable $L^{p}$-integrability condition on the distortion function as long as the Sobolev regularity of the mapping is high enough, see again \cite{HenclKoskela}. For
an overview on this topic, see Hencl and Rajala \cite{HenclRajala-Survey}. 

When we study branched covers it is natural to ask whether we can describe what are the allowable branch sets for this class of mappings. For this, one would need to provide, at least, a comprehensive description on the topological and geometric structure of these sets. In dimension two this question is quite well understood: by the classical Sto\"ilow theorem, see e.g.\ \cite{Stoilow} or \cite{LuistoPankka-Stoilow}, the branch set of a planar branched cover is a discrete set. In higher dimensions the Chernavski\u{\i}-V\"ais\"al\"a theorem, see \ \cite{Ch1, Ch2, Vaisala}, states that the branch set of a branched cover between two $n$-manifolds has topological dimension of at most $n-2$. The 
sharpness of this result is not known in all dimensions and even the following
conjecture of Church and Hemmingsen from \cite{ChurchHemmingsen1} remains open: 
\begin{flushleft}
\emph{The branch set of a branched cover between 3-manifolds has
topological dimension one.} 
\end{flushleft}
The solution to this conjecture by itself would not provide us a description on the allowable branch sets, but it would narrow down
the possible behaviour of branched covers.
Partial results and new approaches in this topological setting have been recently obtained, for instance, in the work by Aaltonen and Pankka \cite{AaltonenPankka}.

Even though, for mappings of finite distortion, and especially for quasiregular mappings, more 
properties of the branch set are known, much of the finer details
remain unsolved and are subject to much interest in the field. Especially, the above mentioned question on the topological and geometric structure of the branch set has remained open also for these mapping classes.
In this paper we study the question of allowable branch sets
in the context of entire quasiregular mappings and branched covers with finite distortion between Euclidean spaces of dimension $n \ge 3$. In addition, we will provide certain Zorich-type global homeomorphism theorems for these mapping classes. The main motivation and inspiration to our study comes from the following question of Heinonen which he stated in his ICM address \cite[Section 3]{HeinonenICM}: 
\begin{flushleft}
\emph{Can we describe the geometry and the topology of the allowable branch sets of quasiregular
mappings between metric $n$-manifolds?}
\end{flushleft}

Our first result says that under a sublogarithmic bound on the outer distortion function the
branch set of a branched cover between Euclidean spaces has to be either empty or unbounded. In particular, it follows from this result that compact sets are not allowable branch sets for branched covers with a sublogarithmically growing outer distortion function.

\begin{theorem}\label{thm:Main}
  Let $f \colon \R^n \to \R^n$ be a branched cover with $n \ge 3$.
  Suppose that for every constant $C > 0$ there exists a radius $r>0$
  such that the restriction $f_{r} \colonequals f|_{\R^n \setminus \overline{B^n}(0,r)}$ of the mapping $f$ is a mapping of finite distortion
  that satisfies
  \begin{align*}
    K_O(x,f_r) \leq C\log(|x|) \quad \text{ for a.e.\ }x \in \R^n \setminus \overline{B^n}(0,r).
  \end{align*}
  Then the branch set of $f$ is either empty or unbounded.
\end{theorem}
Note that a mapping satisfying the assumptions of Theorem~\ref{thm:Main} with an empty branch set is actually a global homeomorphism. We explain the details of this Zorich-type global homeomorphism result in Remark \ref{rmk:GlobalHomeomorphicity}. We also point out that the assumption of the mapping being a branched cover in Theorem~\ref{thm:Main} can be replaced by a suitable integrability condition on the distortion function. For further discussion in this direction, see Section~\ref{MFDSectionOpenAndDiscrete}. 

As every entire quasiregular mapping satisfies the assumptions of Theorem~\ref{thm:Main} we also get the following corollary which says that the nonempty, compact sets are not allowable branch sets for entire quasiregular mappings in dimension $n \ge 3$: 
\begin{corollary}\label{coro:Main}
  Let $f \colon \R^n \to \R^n$ be a quasiregular mapping with $n \ge 3$. Then the branch set $B_f$ is either
  empty or unbounded.
\end{corollary}
It seems that Corollary~\ref{coro:Main} has gone unnoticed in the literature. However,
by using a M\"{o}bius transformation that reflects the space with respect to the unit sphere
the existence of a compact branch set translates into a question about the removability of singularities for local homeomorphisms.
By arguing this way, Corollary~\ref{coro:Main} can be deduced also from a theorem of Martio, Rickman and V\"ais\"al\"a \cite[Theorem 3.15]{MRV}.

In Section \ref{sec:construction} we show that Theorem \ref{thm:Main} is sharp
in the logarithmic scale. We record the construction of our example as the following theorem:
\begin{theorem}\label{thm:Example}
  Let $n \ge 3$ and $\varepsilon > 0$.
  Then there exists a piecewise smooth branched cover
  $f \colon \R^n \to \R^n$ of finite distortion and a constant $C \ge 1$ such that 
  \begin{align*}
  K_O(x,f) \le C \max\{1,\log^{1+\eps}(|x|)\} \quad \text{for all $x \in \mathbb{R}^n$},
  \end{align*} 
  and such that the branch set of $f$ is homeomorphic to the $(n-2)$-dimensional torus
  $(\bS^1)^{n-2} \times \{ 0 \} = \bS^1 \times \cdots \times \bS^1 \times \{ 0 \}$.
\end{theorem}
Note that by Theorem~\ref{thm:Example} the topological $(n-2)$-dimensional tori are allowable branch sets for entire branched covers. Furthermore, by modifying the construction one can create a large class of other type of branch sets for branched covers between Euclidean spaces, see Section~\ref{FinalRemarks}. One could think of that this kind of examples of allowable branch sets would arise easily by modifying some known examples of branched covers between spheres (i.e.\ branched covers from $\bS^n$ to $\bS^n$). However, it seems that in terms of the complexity of the branch sets there are more examples of branched covers between spheres than between Euclidean spaces, see e.g.\ \cite{HeinonenRickman-S3, HeinonenRickman, Rickman85}. Especially, it turns out that the examples constructed between spheres do not easily yield examples of similar type of mappings between Euclidean spaces; we will elaborate on this topic in Remark \ref{remark:Counter}. From this point of view, one aim with this note is to expand the collection of examples of branched covers with new type of allowable branch sets.

Finally, we want to mention the connection of our results to two well-known open questions from the theory of quasiregular mappings. The first question goes back to Heinonen and Rickman \cite{HeinonenRickman-S3} and is the following:
\begin{flushleft}
\emph{Let $f \colon \bS^3 \to \bS^3$ be a branched cover. Does there exist homeomorphisms $h_1,h_2 \colon \bS^3 \to \bS^3$ such that $h_1 \circ f \circ h_2$ is a quasiregular mapping?}
\end{flushleft}
For Euclidean spaces the corresponding question is trivial. Indeed, by the Zorich theorem,
see \cite[Corollary III.3.8]{Rickman-book}, an entire quasiregular mapping with an empty branch set
is a global homeomorphism. On the other hand, it is easy to construct a locally homeomorphic, entire branched cover which is not globally injective, see e.g.\ Remark \ref{rmk:GlobalHomeomorphicity}. The mapping constructed in Section \ref{sec:construction}, together with Theorem \ref{thm:Main} does give, 
however, a more involved example. Indeed, our construction gives a branched cover of finite distortion between Euclidean spaces which is not 
reparametrizable to a quasiregular mapping, even though, the restriction of the mapping to every open ball defines a quasiregular mapping.

The second question we want to mention is the well-known open problem of Vuorinen (see \cite[Remarks 3.7]{Vuorinen2} and \cite[p.193, (4)]{Vuorinen}) which asks if a compact set can be an allowable branch set for proper quasiregular mappings defined on an open ball of dimension $n \ge 3$:
\begin{flushleft}
\emph{Let $f \colon B^n(0,1) \to \R^n$ be a proper quasiregular mapping with $n \geq 3$. Can the branch set $B_f$ of such a mapping $f$ be compact and non-empty?}
\end{flushleft}
The construction we give in Theorem~\ref{thm:Example} easily yield a quasiregular mapping of an open ball with a compact branch set. However, it is easy to check that this restriction is not a proper mapping. In addition, as \cite[Remarks 3.7]{Vuorinen2} suggest there is not an easy way to fix this problem as the boundary behaviour of a possible counter example has to be somehow pathological. Therefore, the question of Vuorinen remains open.

\section{Preliminaries}\label{sec:Preli}

\subsection{Notation}
A point $x \in \mathbb{R}^n$ is denoted in coordinates by $(x_1, x_2, \ldots, x_n)$ and its Euclidean norm is denoted by $\norm{x} \colonequals \sqrt{\sum_{i=1}^n x_i^2}$. We denote by $\overline{\R^n}$ the one point compactification of $\R^n$ and we identify this set with the $n$-dimensional \emph{sphere} 
$$\bS^n \colonequals \{ x \in \mathbb{R}^{n+1} : \norm{x} = 1 \}.$$ 
In addition,
$$B^{n}(a,r) \colonequals \{ x \in \mathbb{R}^n : \norm{x-a} < r \}$$
denotes the $n$-dimensional \emph{open ball} of radius $r>0$, centered at $a \in \mathbb{R}^n$. The corresponding \emph{closed ball} is denoted by
$$\overline{B^{n}}(a,r) \colonequals \{ x \in \mathbb{R}^n : \norm{x-a} \le r \}.$$
We will omit the superscript $n$ whenever the dimension is clear from context.
More generally, for an arbitrary set $E \subset \mathbb{R}^n$ we denote by $\overline{E}$ the closure of $E$ and $\partial E$ denotes the topological boundary of this set. 

A positive constant which is depending only on the parameters $p_1, \ldots, p_k$ is denoted by $C \colonequals C(p_1, \ldots, p_k)$. The constant $C$ might change from line to line. Furthermore, for a given function $g \colon \mathbb{R}^n \to [0, \infty)$ we denote
\begin{align*}
o(g) \colonequals \{ f \colon  \mathbb{R}^{n} &\to [0, \infty) : \text{for every $\varepsilon > 0$ there is $C(\varepsilon)\ge 0$} \\
&\text{such that $f(x) \le \varepsilon g(x)$ whenever $\abs{x} \ge C(\varepsilon)$} \}.
\end{align*}
If $f \in o(g)$, we also denote $f \sim o(g)$ or $f = o(g)$.

The \emph{restriction} of a mapping $f \colon  D \to \tilde{D}$ (from $D$ to $\tilde{D}$) to a given set $E \subset D$ is denoted by $f|_E$. 

\subsection{Branched covers and branch sets}

A mapping $f \colon  \Omega \to \mathbb{R}^n$, which is defined on an open set $\Omega \subset \mathbb{R}^n$ with $n \ge 2$, is called
\begin{itemize}
\item \emph{open} if it maps open sets in $\Omega$ to open sets in $\mathbb{R}^n$.
\item \emph{discrete} if for every $y \in \mathbb{R}^n$ the set $f^{-1}(y)$ of pre-images is a discrete set in $\Omega$.
\item \emph{branched cover} if it is continuous, discrete and open.
\item \emph{entire} if the mapping is defined in the whole space, i.e., $\Omega = \mathbb{R}^n$.
\item \emph{proper} if the pre-image of every compact set is a compact set for the mapping $f$.
\end{itemize}
A point $x \in \Omega$ is called a \emph{branch point} of the mapping $f$ if for every open set $U \subset \Omega$ which contains the point $x$ the restriction mapping $f|_U : U \to f(U)$ fails to be a homeomorphisms, in other words, the mapping $f$ fails to be a local homeomorphism at $x$. Note that for continuous and open mappings non-homeomorphicity occurs precisely when the mapping fails to be locally injective. The set
\begin{align*}
B_f \colonequals \{ x \in \Omega : \text{$x$ is a branch point of $f$} \}
\end{align*}
of all branch points of $f$ is called the \emph{branch set} of $f$.

\subsection{Mappings of finite inner distortion}\label{MFDSection}
For the definition and basic properties of  Sobolev spaces we refer to \cite{EvansGariepy}.
A mapping $f \in W_{\loc}^{1,1}(\Omega, \mathbb{R}^n)$, defined on an open set $\Omega \subset \mathbb{R}^n$ with $n \ge 2$, is called a \emph{mapping of finite inner distortion} if
\begin{itemize}
\item[(FI-1)] $J_f \in L_{\loc}^1(\Omega)$, and
\item[(FI-2)] $J_f > 0$ a.e. on the set where $\abs{D^{\sharp}f(x)} > 0$,
\end{itemize}
where $D^{\sharp}f(x)$ stands for the \emph{adjugate matrix} of the matrix $Df(x)$, i.e., the transpose of the cofactor matrix of $Df(x)$, and $\abs{D^{\sharp}f(x)}$ denotes the operator norm of $D^{\sharp}f(x)$.

With every mapping $f \colon  \Omega \to \mathbb{R}^n$ of finite inner distortion we associate a Lebesgue measurable \emph{inner distortion function} $K_I(\cdot, f)\colon \Omega \to [1, \infty]$ defined as follows
\begin{align*}
&K_{I}(x,f) = \left\{ \begin{array}{ll}
\frac{\abs{D^{\sharp}f(x)}^{n}}{J_{f}(x)^{n-1}}, & \textrm{if $J_{f}(x)>0$}\\
1, & \textrm{otherwise.}
\end{array} \right.
\end{align*}
We point out that
\begin{itemize}
\item[(1)] in the planar case the classes of mappings of finite outer distortion and mappings of finite inner distortion coincide. Especially, in the planar case we have 
$$K_I(x,f) = K_O(x,f)$$ 
for almost every $x \in \Omega$.
\item[(2)] in the higher dimensions (i.e. $n \ge 3$) every mapping of finite outer distortion is a mapping of finite inner distortion but another direction is not true. Indeed, if $n \ge 3$ we may define $f \colon  \mathbb{R}^n \to \mathbb{R}^n$ by the formula
$$f(x_1,x_2, \ldots, x_n) = (x_1,0, \ldots, 0).$$
This mapping is obviously a mapping of finite inner distortion but not a mapping of finite outer distortion.
\item[(3)] For every mapping $f \colon  \Omega \to \mathbb{R}^n$ of finite outer distortion we have the following pointwise inequalities
\begin{align*}
K_{I}(x,f) \le K_O(x,f)^{n-1} \quad \text{and} \quad K_O(x,f) \le K_{I}(x,f)^{n-1}
\end{align*}
for almost every $x \in \Omega$.
\end{itemize}
For more details on the properties listed above, see e.g. \cite{HenclKoskela, IwaniecMartin}.
 
\subsection{Continuity, discreteness, and openness of mappings of finite distortion}\label{MFDSectionOpenAndDiscrete} 
Many of the regularity questions of mappings of finite distortion depend
on the level of the integrability of the distortion functions. Next we discuss about the sufficient conformality conditions for the continuity, discreteness, and openness of mappings of finite distortion. For this purpose, we recall the following definition from the literature, see e.g. \cite{KauhanenKoskelaMalyOnninenZhong, KoskelaOnninen-CapacityAndModulus}.

\begin{definition}[Condition $(A)$]
Let $\Omega \subset \mathbb{R}^n$ be a domain (i.e. an open, connected set) with $n \ge 2$. We say that a mapping $f \colon  \Omega \to \mathbb{R}^n$ of finite distortion \emph{satisfies the condition} $(A)$ if
\begin{align*}
\exp(\mathcal{A}(K_{O}(\cdot, f))) \in L_{\loc}^1(\Omega)
\end{align*}
for some continuously differentiable function $\mathcal{A} \colon [0,\infty) \to [0,\infty)$ with 
$$\mathcal{A}(0) = 0 \qquad \text{and} \quad \lim_{t \to \infty} \mathcal{A}(t) = \infty,$$ 
and such that the following conditions hold
\begin{enumerate}
\item[(A-1)] $\int_0^\infty \frac{\mathcal{A}'(t)}{t} \; dt = \infty$, and
\item[(A-2)] there exists $t_0$ such that $\mathcal{A}'(t)t$ increases to infinity for $t \geq t_0$.
\end{enumerate}
\end{definition}
Note that under the assumption of Theorem \ref{thm:Main}, the distortion of $f$ will
be locally bounded outside of a large ball. This especially implies that $f$ satisfies condition $(A)$ outside
large balls as we may choose $\mathcal{A}(t) = t$. The following theorem is a special case of 
\cite[Theorem 1.1.]{KauhanenKoskelaMalyOnninenZhong} and implies that in the statement of 
Theorem \ref{thm:Main} we could replace the requirement of $f$ being a branched cover by requiring
that the mapping is of finite distortion satisfying condition $(A)$.
\begin{theorem}\label{thm:MFD-Reshetnyak}
  Let $f \colon \R^n \to \R^n$, $n \ge 2$, be a mapping of finite distortion satisfying condition $(A)$.
  Then $f$ is either constant or a branched cover.
\end{theorem}
As it was suggested in Section~\ref{sec:Intro}, we could replace the condition $(A)$ in Theorem~\ref{thm:MFD-Reshetnyak} by a suitable $L^{p}$-integrability condition on the distortion function as long as the Sobolev regularity of the mapping is sufficiently high, 
see e.g. \cite[Theorem 3.4]{HenclKoskela}.

\subsection{Paths and the modulus of path families}

For an interval $I \subset \R$ and a continuous path $\gamma \colon I \to \R^n$
we denote the image of the path by $\abs{\gamma}$ and its closure by $\overline{\abs{\gamma}}$.
For an entire and continuous mapping $f \colon \mathbb{R}^n \to \mathbb{R}^n$ we say that a 
point $y_0 \in \R^n$ is \emph{an asymptotic value of $f$} if there exists a path
$\beta \colon [0,\infty) \to \R^n$ such that 
$$\lim_{t \to \infty} f(\beta(t)) = y_0 \quad \text{and} \quad \lim_{t \to \infty} \| \beta(t) \| = \infty.$$

For two collection of paths, $\Gamma_1$ and $\Gamma_2$, we denote $\Gamma_1 \leq \Gamma_2$ if
for any path $\alpha \in \Gamma_1$ there exists a path $\beta \in \Gamma_2$ such that 
$\alpha$ is the restriction of $\beta$ to some subinterval of the domain of $\beta$.

The \emph{modulus} of a path family is an outer measure on the space of all 
paths defined in a given space. We will require both the
weighted modulus and the more classical conformal modulus.
We give the definition here and refer
the reader to \cite{Rickman-book, VaisalaBook} for basic properties of the conformal
modulus and to \cite{KoskelaOnninen-CapacityAndModulus} for basic properties of the weighted modulus.
\begin{definition}
  Suppose that $U \subset \mathbb{R}^n$ is a Borel set with $n \ge 2$. For a collection $\Gamma$ of paths 
  $\gamma \colon (0,1) \to U$ 
  and a given nonnegative function $w \in L_{\loc}^{1}(U)$ we set
  the \emph{$w$-weighted modulus of the path family $\Gamma$} to be
  \begin{align*}
    M_w(\Gamma)
    = \inf_{\rho} \int_{U} \rho^n(x) \, w(x) \, dx,
  \end{align*}
  where the infimum is taken over all Borel functions $\rho \colon U \to [0,\infty]$
  such that for the line integral of $\rho$ along every path $\gamma \in \Gamma$ we have 
  $$\int_{\gamma} \rho \geq 1.$$ 

  For $w \equiv 1$ we denote $M(\Gamma) \colonequals M_w(\Gamma)$ and 
  call $M(\Gamma)$ the \emph{conformal modulus of $\Gamma$}.
\end{definition}

For a mapping $f$ of finite distortion the relations between
the modulus of a path family $\Gamma$ and the modulus of its image family $f(\Gamma) \colonequals \{ f \circ \gamma : \gamma \in \Gamma \}$ are strongly influenced by the distortion properties of the mapping and vice versa. For us the crucial result
is the Poletsky inequality which gives a connection between the conformal modulus $M(f(\Gamma))$ and the weighted modulus $M_{K_{I}(\cdot, f)}(\Gamma)$. The following theorem is a special case of
\cite[Theorem 4.1]{KoskelaOnninen-CapacityAndModulus}:
\begin{theorem}\label{thm:MFD-modulus}
  Let $\Omega \subset \mathbb{R}^n$ be a domain with $n \ge 2$. Let $f \colon \Omega \to \R^n$ be a mapping of finite distortion
  satisfying condition (A). Then there exists a constant $C \geq 1$ such that
  \begin{align*}
    M(f(\Gamma))
    \leq C M_{K_I(\cdot, f)}(\Gamma) \, ,
  \end{align*}
  for any family $\Gamma$ of paths in $\Omega$.
\end{theorem}

Finally, an important tool for our study is a theorem by Agard and Marden in \cite{Agard-Marden} which characterizes 
removable isolated singularities of local homeomorphisms through a certain modulus condition.
Following \cite{Agard-Marden} we say that a continuous mapping $f \colon 
B^n(0,1) \setminus \{ 0 \} \to \R^n$
\emph{satisfies the modulus condition at the origin} if 
for any family of paths $\Gamma_0$ in $B^n(0,1)$ such that $0 \in \overline{|\gamma|}$ for all $\gamma \in \Gamma_0$
we have $M(f \Gamma_0) = 0$.
Note that e.g.\ all quasiregular mappings satisfy the modulus condition by the 
Poletsky-V\"ais\"al\"a inequalities; see \cite[Theorems II.8.1 and II.9.1]{Rickman-book}.
The following theorem is from \cite{Agard-Marden}.
\begin{theorem}\label{thm:AgardMarden}
  Let $f \colon B^n(0,1) \setminus \{ 0 \} \to \R^n$ be a local homeomorphism.
  Then $f$ extends as a local homeomorphism to the whole ball $B^n(0,1)$ if
  and only if $f$ satisfies the modulus condition at the origin.
\end{theorem}

\section{Proof of Theorem \ref{thm:Main}}
\label{sec:Main}

Our proof of Theorem \ref{thm:Main} relies on the fact that if a branched cover $f \colon \R^n \to \R^n$,
$n \geq 3$, extends
to a mapping $\hat f \colon \bS^n \to \bS^n$, the branch set of the original mapping cannot
be nonempty and unbounded.
The following result is well known to the experts in the field, but we have not
seen it explicitly stated in the literature. We give a short proof for the convenience
of the reader.
\begin{proposition}\label{prop:PolyBranch}
  Let $n \geq 3 $ and suppose $f \colon \R^n \to \R^n$ is a branched cover
  that extends to a continuous mapping $\hat f \colon \bS^n \to \bS^n$.
  Then the branch set of $f$ is either empty or unbounded.
\end{proposition}
\begin{proof}
  Suppose $B_f \neq \emptyset$.
  Since $f$ is a branched cover, so is the extension $\hat f \colon \bS^n \to \bS^n$.
  The mapping $\hat f$ is a proper branched cover between $n$-spheres,
  so all points outside $\hat{f}B_{\hat f}$ have equal amount of pre-images
  by topological degree theory, see e.g.\ \cite[Proposition I.4.10]{Rickman-book}.
  Since $B_f \neq \emptyset$, the mapping $f$ is not locally injective.
  Thus at least one point in $f(\R^n)$ has at least two pre-images under $f$ and so also under $\hat f$.
  Since $\# \hat f \inv \{\infty\} = \# \{ \infty \} = 1$, this implies
  that $B_{\hat f} \ni \infty$. But as $B_f$ is bounded,
  $\infty$ is then an isolated point of $B_{\hat f}$. This is a contradiction since $B_{\hat f}$ cannot have
  isolated points in dimensions three and above by the classical result of Church and Hemmingsen \cite[Corollary 5.6]{ChurchHemmingsen1}.
\end{proof}

Thus to forbid compact branch sets for a branched cover $\R^n \to \R^n$ it suffices to show that the mapping extends to
$\bS^n \to \bS^n$, i.e., that there are no asymptotic values. Another way to formulate
this is to say that the infinity point needs to be a removable singularity. By using
a M\"{o}bius transformation $\R^n\setminus\{0\} \to \R^n\setminus\{0\}$
that reflects the space with respect to the unit sphere
and the result of Agard and Marden, Theorem \ref{thm:AgardMarden}, we
see that it suffices to study the modulus of $f(\Gamma_\infty)$, where
$\Gamma_\infty$ is the collection of all paths going to infinity.
We record this observation as the following proposition.
\begin{proposition}\label{prop:Mechanical}
  Let $f \colon \R^n \to \R^n$ be a branched cover and 
  $\Gamma_\infty$ the collection of paths in $\R^n$ 
  that are not contained in any compact subset of $\R^n$.
  If $M(f(\Gamma_\infty)) = 0$, then the branch
  set of $f$ is either empty or unbounded.
\end{proposition}
\begin{proof}
  Towards contradiction let us 
  assume that $B_f$ is bounded and non-empty. By linearly rescaling the domain we may assume
  that $B_f \subset B(0,1)$. Let $h$ be the conformal reflection map
  \begin{align*}
    h \colon 
    B(0,1) \setminus \{ 0 \} 
    \to 
    \R^n \setminus \overline{B}(0,1) ,
    \quad 
    h(x) = \frac{x}{\| x \|^2}
  \end{align*}
  and denote
  \begin{align*}
    g \colonequals  (f|_{\R^n \setminus \overline{B}(0,1)}) \circ h
    \colon B(0,1) \setminus \{ 0 \} \to \R^n.
  \end{align*}
  Since $B_f \subset B(0,1)$ and the mapping
  $h$ is a homeomorphism, the mapping $g$ is a local homeomorphism.
  Furthermore, for the collection $\Gamma$ of non-constant paths in $B(0,1)\setminus \{ 0 \}$ containing
  the origin in their closure, $h(\Gamma) \subset \Gamma_\infty$ and so
  \begin{align*}
    M(g(\Gamma)) 
    = M(f(h(\Gamma)))
    \leq M(f(\Gamma_\infty)) 
    = 0.
  \end{align*}
  Thus by Theorem \ref{thm:AgardMarden} the local homeomorphism
  $g$ extends to a local homeomorphism $\hat g \colon B(0,1) \to \R^n$.
  This implies that the original mapping $f$ extends to a continuous mapping
  between $n$-spheres, which is a contradiction with Proposition \ref{prop:PolyBranch}.
  Thus the original claim holds true.
\end{proof}

By the previous Proposition \ref{prop:Mechanical},
in order to prove our main theorem it suffices 
to show that under the assumptions of Theorem \ref{thm:Main}
we have $M(f(\Gamma_\infty))=0$. 
Since mappings of finite inner distortion are of independent interest
to many people in the field, we prove this crucial property 
in the form of the following more general proposition.
\begin{proposition}\label{prop:ModulusIntegration}
  Suppose that a mapping $f\colon \mathbb{R}^{n} \to \mathbb{R}^n$ satisfies one of the following two conditions:
  \begin{itemize}
  \item[(i)] mapping $f$ is a mapping of finite inner distortion with 
    $$K_{I}(\cdot, f) \sim o\bigl(\log^{n-1}(\abs{x})\bigr).$$

  \item[(ii)] mapping $f$ is a mapping of finite distortion with 
  $$K_{O}(\cdot, f) \sim o\bigl(\log(\abs{x})\bigr).$$ 
  
  \end{itemize}
  Then, if the Poletsky inequality
  \begin{align}\label{PoletskyAssumption}
    M(f(\Gamma))
    \leq C M_{K_{I}(\cdot,f)}(\Gamma)
  \end{align}
  holds for $f$ and for every path family $\Gamma$ in $\mathbb{R}^n$ and for some absolute constant $C>0$, we have $M(f(\Gamma_{\infty})) = 0$.
\end{proposition}
\begin{proof}
Let $\varepsilon>0.$
Let us consider an increasing sequence $\{ r_{i} \}_{i=1}^{\infty}$ of radii 
tending to infinity and such that $r_1 > 1$ and $r_{i+1}\geq r_i^2.$ The precise 
values of $r_i$ will be fixed in a moment.

We denote by $\Gamma_{i}$ the path family of all paths $\gamma$ connecting 
$\partial B(0, r_{i})$ to $\partial B(0, r_{i+1})$ in $A_{i} \colonequals 
B(0,r_{i+1}) \setminus \overline{B(0,r_{i})}$, and we consider the admissible 
test function $\rho_{i} \colon \mathbb{R}^n \to [0,\infty]$ for $\Gamma_{i}$ 
defined as follows 
\begin{displaymath}
\rho_{i}(x) = \left\{ \begin{array}{ll}
\bigl(\log \frac{r_{i+1}}{r_{i}}\bigr)^{-1} \frac{1}{\abs{x}}, & \textrm{if 
$r_{i} < \abs{x} < r_{i+1}$}\\
0, & \textrm{otherwise.}
\end{array} \right.
\end{displaymath}
Then we observe the following:
\begin{itemize}
\item[(1)] Under the assumption $(i)$ of Propostion \ref{prop:ModulusIntegration} 
we may choose $r_i$ such that  $$K_I(x,f)\leq 
\frac{\varepsilon}{2^i}\log^{n-1}\abs{x}$$
in $A_i.$
 Then it follows that
\begin{align}\label{ModuliEst}
  M_{K_{I}(\cdot, f)}(\Gamma_{i}) &\le \int_{\R^n} \rho^n_i(x) K_{I}(x,f) \, dx \\
                                  &\le \frac{\varepsilon}{2^i\log^{n}\frac{r_{i+1}}{r_{i}}} 
                                    \int_{r_{i}}^{r_{i+1}}\int_{S^{n-1}(0,t)} \frac{\log^{n-1}t}{t^{n}} \, 
                                    dt \nonumber\\
                                  &=  \frac{C\varepsilon}{2^i\log^n \frac{r_{i+1}}{r_{i}}} \bigl( 
                                    \log^{n}r_{i+1} - \log^{n}r_{i} \bigr) \nonumber\\
&\le  \frac{C\varepsilon}{2^i}, \nonumber
\end{align}
for all large values of $i \in \mathbb{N}$. 

\item[(2)] Similarly, under  the assumption $(ii)$ of Propostion 
\ref{prop:ModulusIntegration} we may choose $r_i$ such that

$$K_{O}(x,f) \le \frac{\varepsilon}{2^i} \log\abs{x}$$ 
for all $x\in A_i$. Thus, by using the pointwise inequality $K_{I}(x,f) \le 
K_{O}(x,f)^{n-1}$ for almost every $x$ and 
imitating the calculations in \eqref{ModuliEst} we get again
\begin{align}\label{ModuliEst2}
M_{K_{I}(\cdot, f)}(\Gamma_{i}) \le M_{K_{O}(\cdot, f)^{n-1}}(\Gamma_{i}) \le  
\frac{C\varepsilon}{2^i},
\end{align}
for all large values of $i \in \mathbb{N}$.
\end{itemize}

Next we observe that 
\begin{align}\label{Minorizing}
  \bigcup_{i =1}^{\infty} \Gamma_{i}
  \leq\Gamma_{\infty}.
\end{align}
Indeed, all the paths in $\Gamma_{\infty}$ are tending to infinity and therefore each path in $\Gamma_{\infty}$ necessarily intersects at least one of the annuli $A_{i}$ in such a way that it goes first inside $A_i$ by intersecting the inner boundary component of $A_i$ and then exists $A_i$ by intersecting the outer boundary component of $A_i$. Now, from \eqref{Minorizing} it is easy to conclude that
\begin{align}\label{Minorizing2}
  f\biggl(\bigcup_{i =1}^{\infty} \Gamma_{i}\biggr)
  \leq f(\Gamma_{\infty}).
\end{align}
Thus, by repeating the proofs of \cite[Theorem 6.4 and Theorem 
6.2.(3)]{VaisalaBook} with the $K_I$-weighted modulus and applying both
the assumption on the Poletsky inequality \eqref{PoletskyAssumption}, 
\eqref{ModuliEst} and \eqref{ModuliEst2} we see that
\begin{align*}
M(f(\Gamma_{\infty})) &\le M\Bigl( f \Bigl( \bigcup_i \Gamma_{i} \Bigr) \Bigr) \le \sum_{i=N}^{\infty} M(f(\Gamma_i)) \\ 
&\le \sum_{i=N}^{\infty} M_{K_{I}(\cdot, f)}(\Gamma_{i}) < C \sum_{i=N}^{\infty} \frac{\varepsilon}{2^{i}} < C \varepsilon.
\end{align*}
Now, as the constant $C>0$ is independent on $\varepsilon> 0$, by letting $\varepsilon$ tend to zero we get $M(f(\Gamma_{\infty})) = 0$ and the claim follows.
\end{proof}

\begin{remark}
  As a special case of  the proof of Proposition 
\ref{prop:ModulusIntegration} 
  we obtain that the weighted modulus $M_w(\Gamma_\infty)$
  is zero when $w \sim o(\log)$. This notion can be expressed
  as saying that \emph{$\R^n$ is $w$-parabolic when $w \sim o(\log)$};
  see \cite{Pankka-MFD} and \cite{HolopainenPankka-MFD}.
\end{remark}

\begin{proof}[Proof of Theorem \ref{thm:Main}.]
  We may assume $f$ is non-constant, since otherwise $B_f = \R^n$.
  Property (i) in the statement of Proposition \ref{prop:ModulusIntegration}
  holds by assumption and Theorem \ref{thm:MFD-modulus}
  guarantees the required Poletsky inequality, so $M(f(\Gamma_\infty))=0$.
  Thus Proposition \ref{prop:Mechanical} implies that the branch set is 
  either empty or unbounded.
\end{proof}

\begin{proof}[Proof of Corollary \ref{coro:Main}.]
  The corollary follows either as a special case of
  Theorem \ref{thm:Main} or by directly combining Proposition \ref{prop:Mechanical}
  with standard modulus estimates for quasiregular mappings.
\end{proof}

\begin{remark}\label{rmk:GlobalHomeomorphicity}
  We note that the techniques of this section furthermore imply that
  under the assumptions of Theorem \ref{thm:Main}, if the branch set $B_f$ is empty,
  then $f$ is a homeomorphism. This an immediate corollary from the fact that 
  the only local homeomorphisms $\bS^n \to \bS^n$,
  $n\geq 2$, are globally injective by basic theory of covering spaces.
  This global injectivity observation is a special case of the
  results in \cite{HolopainenPankka-MFD}, where a Zorich-type
  global homeomorphism theorem is proved under certain
  parabolicity assumptions which are similar
  to our sublogarithmic distortion bounds. 
  
  Note that this observation is also in some sense strict in our setting.
  Define a diffeomorphism $h \colon \R \to (-\infty,0)$
  such that $h(x) = x-5$ for $x \leq 0$ and 
  $h(x) = x\log^{\frac{1+\eps}{n-1}}(x)$ for $x \geq 2$.
  Now the mapping
  \begin{align*}
    \R^n \to (-\infty,0) \times \R^{n-1},
    \quad
    (x_1,x_2,\ldots,x_n) \mapsto (h(x_1),x_2,\ldots,x_n)
  \end{align*}
  has outer distortion bounded by $C\log^{1+\eps}$ and inner distortion 
  bounded by $C \log^{(1+\eps)(n-1)}$,
  but a small locally bilipshitz postcomposition
  can be used to break global injectivity near the hyperplane $x_1 = 0$.
\end{remark}

\section{Branched cover with $\bS^1$ branch set}
\label{sec:construction}

In this section we construct for any $\eps>0$ a 
continuous, open and discrete mapping $\R^n \to \R^n$ which is piecewise smooth,
has a branch set homeomorphic to $(\bS^1)^{n-2}$ and distortion asymptotically 
bounded by $C\log^{1+\eps}$.
For clarity of the exposition we have divided the construction into three parts:
(a) defining a branched cover $F \colon \R^3 \to \R^3$ with a branch set homeomorphic to 
$\bS^1$, (b) studying the distortion properties of $F$, and (c) extending the
construction to all dimensions $n \geq 3$. We begin, however, with a remark on
the differences between constructing examples between Euclidean spaces or
between $n$-spheres.

\begin{remark}\label{remark:Counter}
  For a domain with punctures mappings with compact branch sets are, perhaps surprisingly, 
  easier to construct. For example; we can take the winding map in $\R^3$,
  $(r,\phi,z) \mapsto (r, 2 \phi, z)$, and extend it to a branched cover
  $f \colon \bS^3 \to \bS^3$ with a branch set homeomorphic to $\bS^1$. By applying
  a M\"obius map we can modify the mapping to have branch set at the 'equator' 
and
  such that $f \inv \{ \infty \} = \{0,\infty\}$. The restriction of this mapping
  to $\bS^3 \setminus \{ 0,\ \infty\}$ yields a surjective quasiregular mapping
  $g \colon \R^3\setminus\{0\} \to \R^3$ with branch set homeomorphic to $\bS^1$.
  Such mapping $g$ cannot, however, be modified in a small neighbourhood
  of the origin to produce a branched cover $\R^3 \to \R^3$ with a compact branch.
  For clarity, we extract this topological observation into Lemma \ref{lemma:Counter}.
\end{remark}

By a simple pole for a continuous mapping $f \colon \Omega \to \R^n$ we mean 
a point $a_0 \in \partial \Omega$ for which $\lim_{x \to a_0} \| f(x) \| = \infty$.
\begin{lemma}\label{lemma:Counter}
  Let $f \colon \R^n\setminus\{a_1, \ldots, a_k\} \to \R^n$ be a branched cover
  with simple poles at the points $a_1, \ldots, a_k$. Then there exists a radius $r_0 > 0$
  such that there is no branched cover $g \colon \R^n \setminus \{ a_2, \ldots, a_k \} \to \R^n$ that agrees 
  with $f$ on the set $\R^n \setminus B(a_1,r_0)$.
\end{lemma}
\begin{proof}
  Since all of the points $a_1, \ldots, a_k$ are simple poles, 
  the mapping $f$ extends naturally to a mapping $\hat f \colon \bS^n \to \bS^n$
  such that 
  \begin{align*}
    \hat f \inv \{ \infty \} 
    = \{ \infty, a_1, \ldots, a_k \}.
  \end{align*}
  For brevity, we denote $a_0 \colonequals \infty$.
  The mapping $\hat f$ is a branched cover, so there exists
  by \cite[Lemma I.4.9]{Rickman-book} arbitrarily small radii $r > 0$ such that
  \begin{align*}
    \hat f \inv B(\infty,r)
    = \bigcup_{j=0}^{k} U(a_j,\hat f,r),
  \end{align*}
  where $U(x,\hat f,r)$ denotes the $x$-component of the set 
  $\hat f \inv B(f(x),r)$. Furthermore by
  \cite[Lemma I.4.9]{Rickman-book} we may assume that the domains
  $U(a_j,\hat f, r)$ disjoint and that they
  are normal domains,
  i.e.\ 
  \begin{align*}
    \hat f \partial U(a_j,\hat f r) 
    =\partial \hat f U(a_j,\hat f r)     
    = \partial B(\infty,r)
  \end{align*}
  for $j = 1, \ldots, k$.
  Excepting $a_0$, these normal domains $U(a_j, \hat f, r)$ in $\bS^n$ correspond to neighbourhoods
  $\hat U_j \subset \R^n$ of the poles of the original mapping $f$. We choose $r_0 > 0$ to be such that
  $B(a_j, r_0) \subset \hat U_j$ for each $j=1, \ldots, k$. 

  Towards contradiction suppose there exists a branched cover
  \begin{align*}
    g \colon \R^n \setminus \{ a_2, \ldots, a_k \} \to \R^n
  \end{align*}
  that agrees with $f$ on the set $\R^n \setminus B(a_1,r_0)$ 
  and denote $U_1 \colonequals U(a_1,\hat f , r)$. Then 
  $g$ also extends as a mapping $\hat g \colon \bS^n \to \bS^n$
  and agrees with $\hat f$ on the boundary $\partial U_1$
  which maps under $\hat f$, and thus under $\hat g$, onto $\partial B(\infty,r)$.
  Since the mapping $\hat g$ is open, $\hat g U_1 \cap B(\infty,r) \neq \emptyset$.
  On the other hand, also by openness of $\hat g$, since no interior point of $U_1$ can be mapped
  into a boundary point of the image $gU_1$, we have $\partial \hat gU_1 = \partial B(\infty,r)$.
  This implies that $\hat g|_{U_1} \colon U_1 \to B(\infty,r)$ is surjective, 
  so especially $\infty \in \hat g U_1$
  which is a contradiction since the mapping $g$ has no poles in $U_1$.
  Thus the original claim holds true.
\end{proof}

\subsection{Construction in $\R^3$}
\label{sec:3Dconstruction}

We construct first the mapping $F$ in three dimensions. The mapping $F$ is basically a sectorial winding around $\mathbb{S}^{1}$. In order to calculate the distortion we need to give an explicit construction. 

By $T_\alpha$ we denote for each $\alpha \in [0,2\pi)$ the half plane forming 
angle $\alpha$ with the plane $T_0=\{(x,0,z): x\geq0\}.$
The mapping $F$ will map each half-plane $T_\alpha$ onto itself, i.e.\
$F T_\alpha = T_\alpha$ for all $\alpha \in [0,2\pi)$,
and the restrictions $F|_{T_\alpha}$ will be topologically equivalent to the complex winding map $z \mapsto z^2$.
We define our mapping on each of the closed half-planes $\overline{T_\alpha}$. For simplicity,
we canonically identify each $\overline{T_\alpha}$, $\alpha \in [0,2\alpha)$, 
with 
\begin{align*}
  T = \{ (x,y) \in \R^2 \mid x \geq 0\}
\end{align*}
such that the positive $z$-axis of $\R^3$ within
$\overline{T_\alpha}$ corresponds to the positive $y$-axis of the plane within $T$.
With this identification the restrictions $F|_{\overline{T_\alpha}}$, $\alpha \in 
[0,2\pi)$ will all be equal
and we denote any and all of the restrictions as $f \colon T \to T$.

The properties of the mapping $f$ are characterized by two homeomorphisms
\begin{align*}
  H \colon [1,\infty) \to (0,1],
  \quad \text{ and }
  \quad
  R \colon [1,\infty) \to [1,\infty),
\end{align*}
with $H(1)=R(1)=1$.
For $f$ to be a branched cover it suffices to to have any such mappings, e.g.\ $R(t) = t$ and $H(t) = t \inv$
for all $t \in [1, \infty)$; see however Section \ref{sec:FiniteDistortion}
for finer properties of $F$ that depend on more subtle choices for the control functions
$H$ and $R$.

On the half space $T$ we fix the open cone 
\begin{align*}
  C 
  \colonequals 
  \left\{
  (x,y) \in T : |y|+1 < x
  \right\}
\end{align*}
which contains the part $(1,\infty)$ of the $x$-axis,
see Figure \ref{fig:construction}. The complement of this cone in $T$ can be expressed as a
union of horizontal line segments, 
\begin{align*}
  T \setminus C 
  = \bigcup_{t \in \R} [0,1+|t|] \times \{t\},
\end{align*}
and we define $f$ on these line segments affinely such that
\begin{align*}
  [0,1+|t|] \times \{ t \} \mapsto [0,R(1+|t|)] \times \{\operatorname{sgn}(t)( R(1+|t|)-1 )\},
\end{align*}
and so that the $y$-axis is mapped onto itself.
The inside of the cone has a stratification as vertical line segments $I_r$, $r>1$, with endpoints on the boundary
of the cone and intersecting the positive $x$-axis at $(r,0)$. We fix inside the cone
two smaller cones which divide each line segment $I_r$ into five subintervals in equal ratios.
These line segments are mapped affinely onto segments forming a rectangle as in Figure 
\ref{fig:construction} with the first and fifth part mapped 
on top of the interval $I_{R(r)}$ and the third interval intersecting the $x$-axis at $(H(r),0)$.

\begin{figure}[h!]
  \centering
  \resizebox{\textwidth}{!}{
    \begin{tikzpicture}

      \begin{scope}[scale=1.3]

        \begin{scope}[shift={(-5,0)}]
          
          \coordinate(a) at (4,3);
          \coordinate(b) at (4,-3);
          
          \coordinate(A) at (4, 9/5);
          \coordinate(B) at (4,-9/5);

          \coordinate(AA) at (4, 3/5);
          \coordinate(BB) at (4,-3/5);

          \draw[fill,black!10] (-0.5,-3)--(-0.5,3)--(a)--(1,0)--(b)--(-0.5,-3);
          \draw[dashed] (-0.5,-3) -- (-0.5,3);
          \draw (-0.5,0) -- (4,0);

          \draw[fill] (-0.5,0) circle [radius=0.03];
          \node[below left] at (-0.5,0) {$0$};
          \draw[fill] (1,0) circle [radius=0.03];
          \node[below left] at (1,0) {$1$};

          \draw (1,0) -- (a);
          \draw (1,0) -- (b);

          \draw[dashed] (1,0) -- (A);
          \draw[dashed] (1,0) -- (B);

          \draw[dashed] (1,0) -- (AA);
          \draw[dashed] (1,0) -- (BB);

          \draw[thick,-{latex}] (3,-2  )--(3,-6/5);
          \draw[thick,-{latex}] (3,-6/5)--(3,-2/5);
          \draw[thick,-{latex}] (3,-2/5)--(3,2/5);
          \draw[thick,-{latex}] (3, 2/5)--(3,6/5);
          \draw[thick,-{latex}] (3, 1  )--(3,2);

          \draw[fill] (3,0) circle [radius=0.03];
          \node[below right] at (3,0) {$r$};

        \end{scope}

      \draw[->, thick] (-0.5,1) to [out=10,in=170] (1,1);
      \node [below] at (0.25,1.5) {$f$};

      \begin{scope}[shift={(2,0)}]

          \coordinate(a) at (4,3);
          \coordinate(b) at (4,-3);
          
          \coordinate(A) at (4,1.5);
          \coordinate(B) at (4,-1.5);

          \draw[fill,black!10] (-0.5,-3)--(-0.5,3)--(a)--(1,0)--(b)--(-0.5,-3);
          \draw[dashed] (-0.5,-3) -- (-0.5,3);
          \draw (-0.5,0) -- (4,0);

          \draw[fill] (-0.5,0) circle [radius=0.03];
          \node[below left] at (-0.5,0) {$0$};
          \draw[fill] (1,0) circle [radius=0.03];
          \node[below left] at (1,0) {$1$};

          \draw (1,0) -- (a);
          \draw (1,0) -- (b);

          \draw[thick,-{latex}] (3,-2) -- (3,2);
          \draw[thick,-{latex}] (3,2) -- (0.5,2);
          \draw[thick,-{latex}] (0.5,2) -- (0.5,-2.1);
          \draw[thick,-{latex}] (0.5,-2.1) -- (3.1,-2.1);
          \draw[thick,-{latex}] (3.1,-2.1) -- (3.1,2);

          \draw[fill] (3,0) circle [radius=0.03];
          \node[below right] at (3,0) {$R(r)$};

          \draw[fill] (3,2) circle [radius=0.03];
          \node[right] at (3,2) {$(R(r),R(r)-1)$};
          
          \node[above left] at (0.5,0) {$H(r)$};
          \draw[fill] (0.5,0) circle [radius=0.03];          

        \end{scope}

      \end{scope}
    \end{tikzpicture}
  }
  \caption{Constructing the restrictions $f \colonequals F|_{T_\alpha}$ in the construction of Section \ref{sec:construction}.}
  \label{fig:construction}
\end{figure}
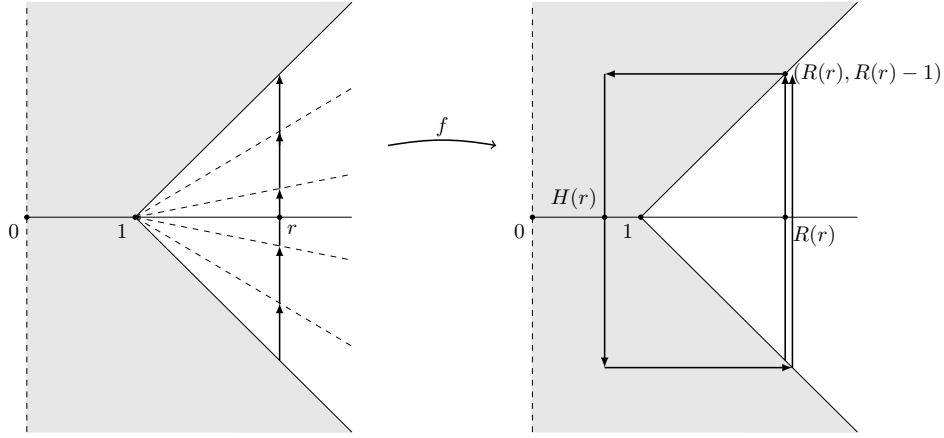

A moment's thought shows that by defining $F$ on each of the closed half-planes $\overline{T_\alpha}$ as above
we receive a continuous, open and discrete mapping which is topologically equivalent
to the winding mapping on each half plane $T_\alpha$, and with branch set
\begin{align*}
  B_F 
  = \{ (x,y,z) \in \R^3 : \| x \|^2 + \| y \|^2 = 1, z = 0 \}
  \simeq \bS^1.
\end{align*}

\subsection{Distortion properties of $F$.}
\label{sec:FiniteDistortion}

The branched cover $F \colon \R^3 \to \R^3$ built in Section \ref{sec:3Dconstruction}
has a compact branch set but by our main results
it is not a quasiregular mapping, or even a mapping of finite distortion satisfying
the distortion assumptions of Theorem \ref{thm:Main}.
However, in this section we show that by setting $R(r) = \log^{\eps/4}(r)$ and $H(r) = \log^{-\eps/4}(r)$,
the distortion of the mapping $F$ satisfies $K_F(x) \leq C_\eps \log^{1+\eps}(|x|)$ for a.e.\ $x \in \R^n \setminus B(0,2)$
and is bounded in $B(0,3)$. Note that even though the restriction of the mapping $F$ to any ball $B(0,r)$ is
quasiregular, these restrictions are not proper for $r>1$ and thus do 
not give a solution to the conjecture of Vuorinen in \cite[p. 193, (4)]{Vuorinen}.

Since the mapping $F$ is symmetric with respect to the half-planes $T_\alpha$,
it suffices to study distortion of $F$ on the half-plane
\begin{align*}
  T_0 
  = (0,\infty) \times \{ 0 \} \times \R,
\end{align*}
i.e.\ it suffices to study the distortion of $F$ on points $(x,y,z)$ such that
$x > 0$ and $y=0$. We denote $f(x,z) = F(x,0,z)$ and consider
$f$ as a mapping from the right half-plane to itself.
We denote its component functions as $f_x$ and $f_z$.

We note first that as $FT_0 \subset T_0$, we have 
\begin{align*}
  D_x F_y (x,0,z) 
  = 0 
  = D_z F_y (x,0,z)
\end{align*}
for all $(x,z) \in \R_+ \times \R$.
Furthermore the symmetric structure of the mapping with respect to the half-planes
$T_\alpha$ implies regularity
for the derivative in the $y$-direction and so
\begin{align*}
  D_y F_x (x,0,z) 
  = 0 
  = D_y F_z (x,0,z)
\end{align*}
for all $(x,z) \in \R_+ \times \R$.
Finally, again by the symmetric structure of $F$,
horizontal circles centered around the $z$-axis 
are mapped in to horizontal circles, more precisely
\begin{align*}
  \{ (x,y,z) : z = t, \| (x,y) \| = r \}
  \overset{F}{\mapsto}
  \{ (x,y,z) :  z = f_z(r,t), \| (x,y) \| = f_x(r,t) \}.
\end{align*}
This implies that $D_y F_y(x,y,z) = \frac{f_x(x,y)}{x}$
for all $(x,z) \in \R_+ \times \R$, see Figure \ref{fig:TopDownView} for details.

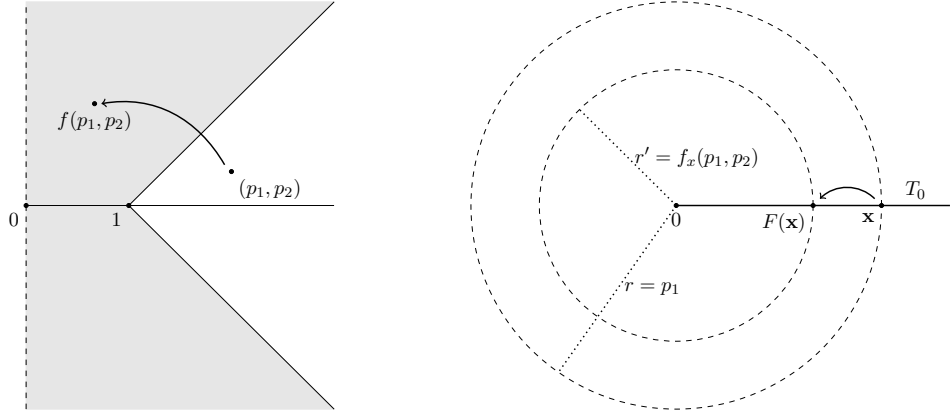
\begin{figure}[h!]
  \centering
  \resizebox{\textwidth}{!}{
    \begin{tikzpicture}

      \begin{scope}[scale=1.3]

        \begin{scope}[shift={(-5,0)}]
          
          \coordinate(a) at (4,3);
          \coordinate(b) at (4,-3);
          
          \draw[fill,black!10] (-0.5,-3)--(-0.5,3)--(a)--(1,0)--(b)--(-0.5,-3);
          \draw[dashed] (-0.5,-3) -- (-0.5,3);
          \draw (-0.5,0) -- (4,0);

          \draw[fill] (-0.5,0) circle [radius=0.03];
          \node[below left] at (-0.5,0) {$0$};
          \draw[fill] (1,0) circle [radius=0.03];
          \node[below left] at (1,0) {$1$};

          \draw (1,0) -- (a);
          \draw (1,0) -- (b);

          \coordinate (XY) at (2.5,0.5);
          \coordinate (fXY) at (0.5,1.5);

          \draw[fill] (XY) circle [radius=0.03];
          \node[below right] at (XY) {$(p_1,p_2)$};
          
          \draw[fill] (fXY) circle [radius=0.03];
          \node[below] at (fXY) {$f(p_1,p_2)$};

          \draw[->,thick] ($(XY) + (-0.1,0.1) $) to [out=120, in=10] ($(fXY) + (0.1,0)$);
        \end{scope}

        \begin{scope}[shift={(4,0)}]
          
          \draw[fill] (0,0) circle [radius=0.03];
          \node[below] at (0,0) {$0$};
          
          \draw[fill] (2,0) circle [radius=0.03];
          \node[below left] at (2,0) {$F(\mathbf{x})$};

        \draw[fill] (3,0) circle [radius=0.03];
        \node[below left] at (3,0) {$\mathbf{x}$};

        \draw[thick] (0,0) -- (4,0);
        \draw[dashed] (0,0) circle [radius=2];

        \draw[dotted,thick] (0,0) -- (135:2);
        \node[right] at (135:1) {$r'=f_x(p_1,p_2)$} ;
        \draw[dashed] (0,0) circle [radius=3];
        \draw[dotted,thick] (0,0) -- (235:3);
        \node[right] at (235:1.5) {$r=p_1$} ;
        
        \node[above] at (3.5,0) {$T_0$};

        \draw[->,thick] ($(3,0) + (-0.1,0.1) $) to [out=135, in=45] ($(2,0) + (0.1,0.1)$);
      \end{scope}

      \end{scope}
    \end{tikzpicture}
  }
  \caption{Behaviour of the mapping $F$ seen within the half-space $T_0$ and from above.}
  \label{fig:TopDownView}
\end{figure}

We now see that by combining the properties
above we have
\begin{align}\label{eq:DF}
  DF(x,y,z) 
  = 
  \begin{bmatrix}
    D_x f_x(x,z) & 0 & D_z f_x(x,z) \\
    0 & \frac{f_x(x,z)}{x} & 0 \\
    D_x f_z(x,z) & 0 & D_z f_z(x,z) \\
  \end{bmatrix},
\end{align}
and so we especially note that
\begin{align*}
  J_F(x,0,z)
  = \frac{f_x(x,z)}{x} J_f(x,z),
  \quad
  \| DF(x,0,z) \|
  \leq \frac{f_x(x,z)}{x} + \| Df(x,z) \|.
\end{align*}
To calculate the distortion of $F$ we note
that in fact $f_z(x,-z) = -f_z(x,z)$ for all
$(x,z) \in \R_+ \times \R$, and so for the distortion
estimates it suffices
to study the case $z > 0$, i.e.\ the upper-right quadrant $U \colonequals (0,\infty)^2$. We divide
the domain $U$ into five parts based on the complement of the cone $C$ and its subdivision,
see again Figure \ref{fig:construction}:
\begin{align*}
  A   &= \{ (x,z) \in U \mid x - 1 < z\} \\
  I^5 &= \{ (x,z) \in U \mid \frac{3}{5}(x-1) < z < x-1  \} \\
  I^4 &= \{ (x,z) \in U \mid \frac{1}{5}(x-1) < z < \frac{3}{5}(x-1)  \} \\
  I^3 &= \{ (x,z) \in U \mid 0 < z < \frac{1}{5}(x-1)  \} \\
  S   &= \{ (x,z) \in U \mid z = \frac{j}{5}(x-1), j =1,3,5\}.\\
\end{align*}
The set $S$ consists of finitely many rays and thus has measure zero, and we
can omit it in our a.e.\ distortion estimates. On the rest four domains
we can calculate an exact expression for the two component functions $f_x$ and $f_z$ of $f$:
\begin{align*}
  f_x(x,z)
  = 
  \begin{cases}
     \frac{x}{z+1} R(z+1) ,& (x,z) \in A \\
     R(x) , & (x,z) \in I^5 \\
     \left( 5\frac{z}{x-1} - 1\right)\frac{R(x)}{2} - 
\left(5\frac{z}{x-1} - 3\right) \frac{H(x)}{2}, & (x,z) \in I^4 \\
     H(x) , & (x,z) \in I^{3} \\
  \end{cases},
\end{align*}
and
\begin{align*}
  f_z(x,z)
  =
  \begin{cases}
     R(z+1) -1 , & (x,z) \in A \\
     (R(x)-1) (5\frac{z}{x-1} - 4) , & (x,z) \in I^5 \\
     -R(x) +1 , & (x,z) \in I^4 \\
     -5 \frac{z}{x-1}(R(x)-1) , & (x,z) \in I^{3} \\
  \end{cases},
\end{align*}
Note that  we have $\frac{x}{z+1} \in (0,1)$ for all $(x,z) \in A$ and
$\frac{z}{x-1} \in (0,1)$ for all $(x,z) \in (0,\infty)^2 \setminus A$.
With the aid of the above exact expressions for $f$ and \eqref{eq:DF} it follows that
\begin{align}\label{eq:ExampleDistortion}
  K_F(x,0,z) 
  \lesssim
  \begin{cases}
    \frac{R(|z|)}{|z| R'(|z|)}, & (x,z) \in A \\
    \frac{R(|x|)}{|x| R'(|x|)}, & (x,z) \in I^5 \\
    \frac{R(|x|)^2}{|x| R'(|x|) H(|x|)}, & (x,z) \in I^4 \\
    \frac{R(|x|)^2}{|x| H'(|x|) H(|x|)}, & (x,z) \in I^{3} \\
  \end{cases}.
\end{align}
Thus, when we set $R(t) = \log^\eps(t)$ and $H(t) = \log^{-\delta}(t)$, for 
any $r \geq 2$
the distortion of $F$ at a point $(x,0,z)$, at thus at any point $(x,y,z)$ is bounded by a constant $K=K(r)$
when $\| (x,y,z) \| \leq r$ and 
bounded by $\log^{1 + 2\eps + 2\delta}(\| (x,y,z) \|)$ 
when $\| (x,y,z) \| > r$.

Finally we remark that since the functions $H$ and $R$ are continuously differentiable, 
so is the mapping $F$ outside the union of six rectifiable surfaces generated
by the rotation of the sets $S, -S \subset T_0$. Furthermore the mapping $F$ is $L(r)$-Lipschitz
in any ball $B(0,r)$.

\subsection{Extension to higher dimensions}

The construction of $F$ in three dimension was based on first defining
a mapping $f\colon T \to T$ from a closed half-space to itself such that
$f\partial T = \partial T$ and then extending this mapping to all of $\R^3$
by symmetry. This basic idea goes through in higher dimensions
as well, and we describe the extension in detail for $n=4$.

\begin{figure}
  \centering
  \resizebox{\textwidth}{!}{
    \begin{tikzpicture}

      \begin{scope}[scale=1.3]

        \begin{scope}[shift={(-6,-3)}]
          
          \coordinate (A) at (0,5);
          \coordinate (B) at (5,0);
          \coordinate (O) at (0,0);

          \coordinate (a) at (2,5);
          \coordinate (b) at (5,2);
          \coordinate (o) at (1,1);          
          \coordinate (c) at (4,4);          
          \node[right] at (c) {$(r,r)$};          
          
          \draw[fill,black!10] (A)--(O)--(B)--(b)--(o)--(a)--cycle;

          \draw[dashed] (1,0)--(o)--(0,1);

          \draw[fill,white] (0.5,0.5) circle [radius=0.3];
          \node at (0.5,0.5) {$\id$};

          \node[left] at (0,1) {$1$};
          \node[below] at (1,0) {$1$};

          \draw[dashed] (A) -- (O) -- (B);
          \draw[dashed] (o) -- (5,5);

          \draw[thick] (a) -- (o) -- (b);

          \draw[<->] (0,4.5)--(1.85,4.5);
          \node[left] at (0,4.5) {$t$};

          \draw[fill] (O) circle [radius=0.03];
          \node[below left] at (O) {$0$};
          \draw[fill] (o) circle [radius=0.03];
          \draw[fill] (c) circle [radius=0.03];
          

          \draw[->,thick] (4,1.75)--(c)--(1.75,4);
        \end{scope}

      \draw[->, thick] (-0.5,1) to [out=10,in=170] (1,1);
      \node [below] at (0.25,1.5) {$f$};

      \begin{scope}[shift={(2,-3)}]
          
          \coordinate (A) at (0,5);
          \coordinate (B) at (5,0);
          \coordinate (O) at (0,0);

          \coordinate (a) at (2,5);
          \coordinate (b) at (5,2);
          \coordinate (o) at (1,1);          
          \coordinate (c) at (3,3);

          \draw[fill,black!10] (A)--(O)--(B)--(b)--(o)--(a)--cycle;

          \draw[dashed] (A) -- (O) -- (B);
          \draw[dashed] (o) -- (5,5);

          \draw[thick] (a) -- (o) -- (b);

          \draw[fill] (O) circle [radius=0.03];
          \node[below left] at (O) {$0$};
          \draw[fill] (o) circle [radius=0.03];
          \draw[fill] (c) circle [radius=0.03];
          \node[right] at (c) {$(R(r),R(r))$};


          \draw[->,thick] 
          (3,1.5)--
          (c)--
          (0.5,3)--
          (0.5,0.5)--
          (3.1,0.5)--
          (3.1,3.1)--
          (1.55,3.1);

          \draw[fill] (0.5,0.5) circle [radius=0.03];
          \draw[dashed] (0,0.5)--(0.5,0.5);
          \node[left] at (0,0.5) {$H(r)$};
          \draw[dashed] (0.5,0)--(0.5,0.5);
          \node[below] at (0.5,0) {$H(r)$};

          \draw[<->] (0,4)--(1.7,4);
          \node[left] at (0,4) {$R(t)$};
      \end{scope}
        
      \end{scope}
    \end{tikzpicture}
  }
  \caption{Modifying the restriction $F|_T$ in order to extend the mapping $F$ to dimension $4$.}
  \label{fig:DimensionN-1}
\end{figure}
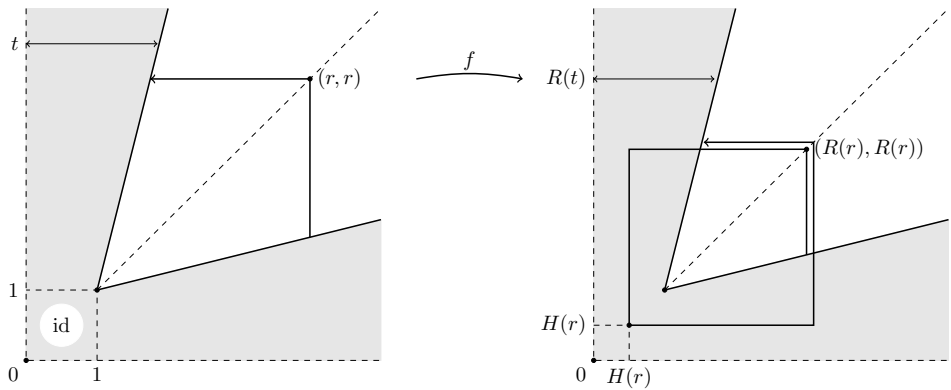

Note that since $\R^3 = \cup_\alpha T_\alpha$, we have
$
\R^4 
= \R \times \cup_\alpha T_\alpha
= \cup_\alpha \R \times T_\alpha,
$
and we may identify $\R \times T_\alpha$ with the half-space
$\{ (x,y,z) \in \R^3 \mid z > 0 \}$. Denote the closure of this half-space by $S$.
To imitate the previous construction we need to define a mapping $g \colon S \to S$ 
with a compact branch set such that $g \colon \partial S \to \partial S$.
One way to achieve this is to modify the 'sector windings' $F|_{T_\alpha}$ to be
defined only in the upper half of the half-space $T$; see Figure \ref{fig:DimensionN-1}.
Call this modification $g$ and set
\begin{align*}
  G \colon \R^4 \to \R^4,  \quad G|_{\R \times T_\alpha} = g.
\end{align*}
Now $G$ is a branched cover with branch set homeomorphic
to $\bS^1 \times B_g \simeq \bS^1 \times \bS^1$.
Furthermore we may imitate the distortion estimates of Section
\ref{sec:FiniteDistortion} for $G$; besides the fact that the sector is stratified
into a different number of line segments the calculations are similar
and yield comparable estimates.

Finally we remark that this procedure can be continued to generate
mappings in all higher dimensions. For the inductive step, to define
a mapping $H \colon \R^{n+1} \to \R^{n+1}$, we merely need to modify
the mapping of the previous step to reside only in a half-space.
This concludes the proof of Theorem \ref{thm:Example}.

\section{Final remarks}\label{FinalRemarks}

The proof of Proposition \ref{prop:ModulusIntegration} depends on the fact that
the distortion of $f$ satisfies 
\begin{align*}
  \frac{K_f(x)}{\log(|x|)} \to 0, \text{ as } |x| \to \infty
\end{align*}
and does not extend for mappings $f$ with $K_f(x) \lesssim \log(|x|)$. On the other hand
it can be shown that any selection of the functions $R$ and $H$ in the construction of
our example can not yield $K_f(x) \lesssim \log(|x|)$. Thus the methods here do not give
any information whether the branch set of a mapping of finite distortion with 
$K_f(x) \simeq \log(|x|)$ can be compact and non-empty.

The mapping $f \colon \R^3 \to \R^3$ constructed in Section \ref{sec:3Dconstruction} 
can be used to construct
for any $N \geq 1$ branched covers of almost logarithmic distortion such that the branch set is 
a disjoint collection of $N$ copies of $\bS^1$. It is also not hard to combine these mappings
to construct branch sets which are homeomorphic to finite collections of circles linked
in $\R^3$. These ideas even extend to constructing torus knots in the branch; with
more than one component of the branch this is straightforward, but with a careful
local argument a single component suffices when the degree of the map is sufficiently large. Indeed,
in \cite[Lemma 3.1 and Theorem 3.2]{ChurchHemmingsen2} Church and Hemmingsen demonstrate
how a knotted branch may be achieved by modifying a 
winding map of sufficiently high degree between spheres near the branch set.
This local modification lends itself to our setting and
the mapping $f \circ f \colon \R^3 \to \R^3$
can be modified in a neighbourhood of the circular branch set to have a knotted branch.
We would be interested to know if a knotted branch set is possible with a branched cover of degree 2.

\section*{\textbf{Acknowledgments}}

We would like to thank Pekkas Koskela and Pankka for helpful remarks.


\newcommand{\etalchar}[1]{$^{#1}$}
\def\cprime{$'$}\def\cprime{$'$}

\end{document}